\documentclass[12pt]{amsart}
\usepackage{amsmath,amsthm,amsfonts,amscd,amssymb,eucal,latexsym,mathrsfs, appendix, yhmath, setspace}
\usepackage[numbers,sort&compress]{natbib}
\usepackage[all,cmtip]{xy}
\usepackage{enumerate}

\newtheorem{theorem}{Theorem}[section]
\newtheorem{corollary}[theorem]{Corollary}
\newtheorem{lemma}[theorem]{Lemma}
\newtheorem{proposition}[theorem]{Proposition}

\theoremstyle{definition}

\newcommand{\htopol}{{\text{\rm h}}_{\text{\rm top}}}

\newcommand{\sep}{{\rm sep}}

\newcommand{\spn}{{\rm span}}

\newcommand{\cM}{{\mathcal M}}

\newcommand{\cP}{{\mathcal P}}

\newcommand{\cU}{{\mathcal U}}

\newcommand{\Nb}{{\mathbb N}}

\newcommand{\Rb}{{\mathbb R}}
\newcommand{\Zb}{{\mathbb Z}}

\allowdisplaybreaks

\begin{document}

\title{Garden of Eden and Specification}

\author{Hanfeng Li}

\address{\hskip-\parindent
Center of Mathematics, Chongqing University,
Chongqing 401331, China.
\break
Department of Mathematics, SUNY at Buffalo,
Buffalo, NY 14260-2900, U.S.A.}
\email{hfli@math.buffalo.edu}


\date{January 5, 2018}

\begin{abstract}
We establish a Garden of Eden theorem for expansive algebraic actions of amenable groups with the weak specification property, i.e. for any continuous equivariant map $T$ from the underlying space to itself, $T$ is pre-injective if and only if it is surjective. In particular, this applies to all expansive principal algebraic actions of amenable groups and expansive algebraic actions of $\Zb^d$ with CPE.
\end{abstract}

\maketitle


\section{Introduction} \label{S-introduction}

Let a countable discrete group $\Gamma$ act on a compact metrizable space $X$ continuously and let $T$ be a $\Gamma$-equivariant continuous map $X\rightarrow X$. In this paper we consider the relation between the surjectivity of $T$ and a weak form of injectivity of $T$.

In the case $X=A^\Gamma$ for some finite set $A$ and $\Gamma$ acts on $A^\Gamma$ by shifts, $T$ is called a cellular automaton \cite{CC10} and can be thought of as an evolution determined  by a local rule. In this case $T$ is not surjective exactly when there is a $w\in A^F$ for some finite set $F\subseteq \Gamma$ such that $w$ is not equal to the restriction of any element of $T(A^\Gamma)$ on $F$. Such a $w$ is called a Garden of Eden (GOE) pattern, meaning that it could  appear only in the first of the sequence $A^\Gamma, T(A^\Gamma), T^2(A^\Gamma), \dots$.  A pair of mutually erasable patterns is a pair $(w_1, w_2)$  of distinct elements of $A^K$ for some finite set $K\subseteq \Gamma$ such that whenever $x_1$ and $x_2$ are elements of $A^\Gamma$ coinciding on $\Gamma\setminus K$ and extending $w_1$ and $w_2$ respectively, one has $T(x_1)=T(x_2)$.
In 1963, Moore \cite{Moore} showed that when $\Gamma=\Zb^d$, if there is a pair of mutually erasable patterns, then there is a GOE pattern.
Soon afterwards Myhill \cite{Myhill} proved the converse. The results of Moore and Myhill were extended to finitely generated groups with subexponential growth by Mach\`{\i} and Mignosi \cite{MM}, and then to all amenable groups by Ceccherini-Silberstein, Mach\`{\i}, and Scarabotti \cite{CMS}. On the other hand, Bartholdi \cite{Bartholdi10, Bartholdi16} showed that the results of Moore and Myhill fail for every nonamenable group.

For general actions of $\Gamma$ on a compact metrizable space $X$, one has the homoclinic equivalence relation defined on $X$ (see Section~\ref{SS-homoclinic}) and $T$ is called {\it pre-injective} if it is injective on each homoclinic equivalence class \cite{Gromov}. For the shift action $\Gamma\curvearrowright A^\Gamma$ with finite $A$, the map $T$ is not pre-injective exactly when there is a pair of mutually erasable patterns. Thus we say the action $\Gamma\curvearrowright X$ has {\it the Moore property} if surjectivity implies pre-injectivity for every $\Gamma$-equivariant continuous map $T:X\rightarrow X$, the action has {\it the Myhill property} if pre-injectivity implies surjectivity for every such $T$, and the action has the {\it Moore-Myhill property} if surjectivity is equivalent to pre-injectivity for every such $T$.
There has been quite some work trying to establish these properties for various actions \cite{CC12, CC15, CC16, CC17a, Fiorenzi00, Fiorenzi03, Gromov}; see also the book \cite{CC10} and the recent survey \cite{CC17b}.

We say an action $\Gamma\curvearrowright X$ is {\it surjunctive} if injectivity  implies surjectivity for every $\Gamma$-equivariant continuous map $T:X\rightarrow X$. Thus the Myhill property implies surjunctivity. Gottschalk's surjunctivity conjecture says that the shift action $\Gamma\curvearrowright A^\Gamma$  for finite $A$ is surjunctive for every group $\Gamma$ \cite{Got}. This was proved for sofic groups by Gromov \cite{Gromov, Weiss, CC10, KL}.

Specification plays a vital role in our work. It is a strong orbit tracing property introduced by Bowen for $\Zb$-actions \cite{Bowen} and extended to $\Zb^d$-actions by Ruelle \cite{Ruelle}. There are a few versions of the specification property \cite[Definition 5.1]{LS} \cite[Definition 6.1]{CL}. For our purpose, the weak specification property (see Section~\ref{SS-specification}) suffices. Our result regarding the Myhill property is the following.

\begin{theorem} \label{T-Myhill1}
Every expansive continuous action of a countable amenable group on a compact metrizable space with the weak specification property has the Myhill property. In particular, such actions are surjunctive.
\end{theorem}

Theorem~\ref{T-Myhill1} strengthens a few known results and is new even for the case $\Gamma=\Zb$. Fiorenzi \cite[Corollary 4.8]{Fiorenzi03} proved Theorem~\ref{T-Myhill1} under the further assumption of subshifts of finite type. (Actually Fiorenzi stated her result for strongly irreducible subshifts of finite type, but strongly irreducible subshifts are exactly subshifts with the weak specification property, see Appendix~\ref{S-subshifts}.) Ceccherini-Silberstein and Coornaert proved Theorem~\ref{T-Myhill1} under the further assumption of subshifts \cite[Theorem 1.1]{CC12} and under the further assumption of a uniformly bounded-to-one factor of a weak specification subshift \cite[Theorem 1.1]{CC15}.

Algebraic actions are actions of $\Gamma$ on compact metrizable abelian groups by automorphisms. The algebraic actions of $\Zb^d$ were studied extensively in 1990s, and the last decade has seen much progress towards understanding the algebraic actions of nonabelian groups; see \cite{Schmidt, KL} and the references therein. Our result concerning the Moore property is the following.

\begin{theorem} \label{T-Moore1}
Every expansive algebraic action of a countable amenable group with completely positive entropy with respect to the normalized Haar measure has the Moore property.
\end{theorem}

If an  algebraic action of a countable  amenable group has the weak specification property, then it has completely positive entropy with respect to the normalized Haar measure (see Corollary~\ref{C-CPE}). For expansive algebraic actions, the converse is also true in the case $\Gamma=\Zb^d$ \cite[Theorem 5.2]{LS}, and is conjectured to hold for polycyclic-by-finite groups \cite[Conjecture 6.5, Theorem 1.2, Corollary 8.4]{CL}. The principal algebraic actions are the actions of the form $\Gamma\curvearrowright \widehat{\Zb\Gamma/\Zb\Gamma f}$ for $f$ in the integral group ring $\Zb\Gamma$ (see Section~\ref{SS-algebraic}).
By Lemma~\ref{L-principal} every expansive principal algebraic action has the weak specification property. Thus combining Theorems~\ref{T-Myhill1} and \ref{T-Moore1} we have

\begin{theorem} \label{T-MM}
Every expansive algebraic action of a countable amenable group with the weak specification property has the Moore-Myhill property. In particular, every expansive principal algebraic action of a countable amenable group and every expansive algebraic action of $\Zb^d$ with completely positive entropy with respect to the normalized Haar measure have the Moore-Myhill property.
\end{theorem}

Theorem~\ref{T-MM} applies to the shift actions $\Gamma\curvearrowright A^\Gamma$ for finite $A$, since one can identify $A^\Gamma$ with $\widehat{\Zb\Gamma/\Zb\Gamma f}$ for $f=|A|$. Previously, Ceccherini-Silberstein and Coornaert established the Moore-Myhill property for hyperbolic toral automorphisms \cite{CC16} and expansive principal algebraic actions of countable abelian groups with connected underlying space \cite{CC17a}.

We remark that even for $\Gamma=\Zb$, not every expansive action with the weak specification property has the Moore property. Indeed, Fiorenzi showed that the even shift in $\{0, 1\}^\Zb$ consisting of all elements with an even number of $0$'s between any two $1$'s does not have the Moore property \cite{Fiorenzi00}.

This paper is organized as follows. In Section~\ref{S-preliminaries} we recall some definitions. Theorems~\ref{T-Myhill1} and \ref{T-Moore1} are proved in Sections~\ref{S-Myhill} and \ref{S-Moore} respectively.  We study the implications of weak specification for combinatorial independence in Section~\ref{S-ind}. The equivalence of weak specification  and strong irreducibility for subshifts is proved in Appendix~\ref{S-subshifts}.

\noindent{\it Acknowledgements.}
The author was partially supported by NSF and NSFC grants. He is grateful to Michel Coornaert for helpful discussion, and to Tullio Ceccherini-Silberstein, David Kerr, Douglas Lind and Chung Nh\^{a}n Ph\'{u} for comments. He also thanks the referee for remarks.

\section{preliminaries} \label{S-preliminaries}

In this section we recall some definitions and set up some notations. Throughout this paper $\Gamma$ will be  a countable discrete group with identity element $e_\Gamma$.

\subsection{Expansiveness and weak specification} \label{SS-specification}

Let $\Gamma$ act on a compact metrizable space $X$ continuously, and let $\rho$ be a compatible metric on $X$.
The action is called {\it expansive} if there is some $\kappa>0$ such that $\sup_{s\in \Gamma}\rho(sx, sy)>\kappa$ for all distinct $x, y\in X$. Such a $\kappa$ is called an {\it expansive constant} with respect to $\rho$.

The action $\Gamma\curvearrowright X$ is said to have {\it the weak specification property} \cite[Definition 6.1]{CL} if for any $\varepsilon>0$ there exists a nonempty symmetric finite subset $F$ of $\Gamma$ such that for any finite collection $\{F_j\}_{j\in J}$ of finite subsets of $\Gamma$ satisfying
$$ F F_i\cap F_j=\emptyset \mbox{ for all distinct } i, j\in J,$$
and any collection of points $\{x_j\}_{j\in J}$ in $X$, there exists a point $x\in X$ such that
$$\rho(sx, sx_j)\le \varepsilon \mbox{ for all } j\in J, s\in F_j.$$
Using the compactness of $X$, it is easy to see that when the action has the weak specification property, one  can actually allow $J$ and $F_j$ to be infinite. It is also easy to see that weak specification passes to factors.

\subsection{Group rings and algebraic actions} \label{SS-algebraic}

We refer the reader to \cite{KL, Schmidt} for detail on group rings and algebraic actions.
The {\it integral group ring} $\Zb\Gamma$ of $\Gamma$ is defined as the set of all finitely supported functions $\Gamma\rightarrow \Zb$, written as $\sum_{s\in \Gamma} f_s s$ with $f_s\in \Zb$ for all $s\in \Gamma$ and $f_s=0$ for all but finitely many $s$, with addition and multiplication given by
\begin{align*}
\sum_{s\in \Gamma} f_s s+\sum_{s\in \Gamma} g_s s&=\sum_{s\in \Gamma} (f_s+g_s)s,\\
(\sum_{s\in \Gamma}f_ss)(\sum_{t\in \Gamma}g_tt)&=\sum_{s, t\in \Gamma}f_sg_t(st).
\end{align*}
The group algebra $\ell^1(\Gamma)$ is the set of all functions $f: \Gamma\rightarrow \Rb$ satisfying $\sum_{s\in \Gamma}|f_s|<+\infty$, with addition and multiplication defined in the same way.

An action of $\Gamma$ on a compact metrizable abelian group by (continuous) automorphisms is called {\it an algebraic action}. Up to isomorphism, there is a natural one-to-one correspondence between algebraic actions of $\Gamma$ and countable left $\Zb\Gamma$-modules  as follows. For any algebraic action $\Gamma\curvearrowright X$, the Pontrjagin dual
$\widehat{X}$ consisting of all continuous group homomorphisms $X\rightarrow \Rb/\Zb$ is a countable abelian group and the action of $\Gamma$ on $X$ induces an action of $\Gamma$ on $\widehat{X}$ which makes $\widehat{X}$ into a left $\Zb\Gamma$-module. Conversely, for any countable left $\Zb\Gamma$-module $\cM$, the Pontrjagin dual $\widehat{\cM}$ consisting of all group homomorphisms $\cM\rightarrow \Rb/\Zb$ under the pointwise convergence topology forms a compact metrizable abelian group and the $\Zb\Gamma$-module structure of $\cM$ gives rise to an action of $\Gamma$ on $\cM$ which induces an action of $\Gamma$ on $\widehat{\cM}$ in turn.

For each $f\in \Zb\Gamma$, the associated algebraic action $\Gamma\curvearrowright \widehat{\Zb\Gamma/\Zb\Gamma f}$ is called a {\it principal algebraic action}.

\begin{lemma} \label{L-principal}
Every expansive principal algebraic action has the weak specification property.
\end{lemma}
\begin{proof} For any $f\in \Zb\Gamma$, the principal algebraic action $\Gamma\curvearrowright \widehat{\Zb\Gamma/\Zb\Gamma f}$ is expansive exactly when $f$ is invertible in $\ell^1(\Gamma)$ \cite[Theorem 3.2]{DS}. If $f\in \Zb\Gamma$ is invertible in $\ell^1(\Gamma)$, then the principal algebraic action $\Gamma\curvearrowright \widehat{\Zb\Gamma/\Zb\Gamma f}$ has the weak specification property \cite[Theorem 1.2]{Ren}.
\end{proof}

\subsection{Homoclinic pairs} \label{SS-homoclinic}

Let $\Gamma$ act on a compact metrizable space $X$ continuously, and let $\rho$ be a compatible metric on $X$. We say a pair $(x, y)\in X^2$ is {\it homoclinic} or {\it asymptotic} if $\rho(sx, sy)\to 0$ as $\Gamma \ni s\to \infty$. The set of all homoclinic pairs is an  equivalence relation on $X$, and does not depend on the choice of $\rho$.
A map from $X$ to another space is called {\it pre-injective} if it is injective on every homoclinic equivalence class.

Now assume that $\Gamma\curvearrowright X$ is an algebraic action. We can always choose $\rho$ to be translation-invariant. It follows that the homoclinic equivalence class of the identity element $0_X$
is a $\Gamma$-invariant subgroup of $X$, which we shall denote by $\Delta(X)$. Furthermore, for any $x\in X$, its homoclinic equivalence class is exactly $x+\Delta(X)$. The following lemma  will be crucial for the proof of Theorem~\ref{T-Moore}.

\begin{lemma} \label{L-metric}
Let $\Gamma\curvearrowright X$ be an expansive algebraic action. Then there is a translation-invariant compatible metric $\rho$ on $X$ such that $\sum_{s\in \Gamma}\rho(sx, 0_X)<+\infty$ for all $x\in \Delta(X)$.
\end{lemma}
\begin{proof} Since the algebraic action $\Gamma\curvearrowright X$ is expansive, by \cite[Theorem 5.6]{CL} \cite[Theorem 13.31]{KL} we have $\Delta(X)=\Delta^1(X)$, where $\Delta^1(X)$ is the $1$-homoclinic group of $X$ defined in \cite[Definition 5.1]{CL} \cite[Definition 13.26]{KL}.

Again using the expansiveness of the algebraic action $\Gamma\curvearrowright X$, by
\cite[Proposition 2.2, Corollary 2.16]{Schmidt}  or \cite[Lemma 13.6]{KL} we know that $\widehat{X}$ is a finitely generated $\Zb\Gamma$-module. Then by \cite[Proposition 5.7]{CL} \cite[Lemma 13.33]{KL} there is a translation-invariant compatible metric $\rho$ on $X$ such that $\sum_{s\in \Gamma}\rho(sx, 0_X)<+\infty$ for all $x\in \Delta^1(X)$.
Consequently, $\sum_{s\in \Gamma}\rho(sx, 0_X)<+\infty$ for all $x\in \Delta(X)$.
\end{proof}

\subsection{Amenable groups and entropy} \label{SS-entropy}

We refer the reader to \cite{MO, CC10, KL} for details on amenable groups and the entropy theory of their actions.
A countable group $\Gamma$ is called {\it amenable} if it has a left F{\o}lner sequence, i.e. a sequence $\{F_n\}_{n\in \Nb}$ of nonempty finite subsets of $\Gamma$ satisfying
$$ \lim_{n\to \infty}\frac{|KF_n\Delta F_n|}{|F_n|}=0$$
for all nonempty finite subsets $K$ of $\Gamma$.

Let $\Gamma$ act on a compact metrizable space $X$ continuously. For a finite open cover $\cU$ of $X$, we denote by $N(\cU)$ the minimal number of elements of $\cU$ needed to cover $X$.
Then the limit $\lim_{n\to \infty}\frac{1}{|F_n|}\log N(\bigvee_{s\in F_n}s^{-1}\cU)$ exists and does not depend on the choice of the F{\o}lner sequence $\{F_n\}_{n\in \Nb}$. We denote this limit by $\htopol(\cU)$. The {\it topological entropy of the action $\Gamma\curvearrowright X$} is defined as
$$\htopol(X):=\sup_{\cU}\htopol(\cU),$$
where $\cU$ runs over all finite open covers of $X$.

Let $\rho$ be a compatible metric on $X$, and  let $\varepsilon>0$. A set $Z\subseteq X$ is called {\it $(\rho, \varepsilon)$-separated} if
$\rho(x, z)\ge \varepsilon$ for all distinct $x, z\in Z$. Denote by $\sep(X, \rho, \varepsilon)$ the maximal cardinality of $(\rho, \varepsilon)$-separated subsets of $X$.
A set $Z\subseteq X$ is called {\it $(\rho, \varepsilon)$-spanning} if
for any $x\in X$ there exists some $z\in Z$ with $\rho(x, z)< \varepsilon$. Denote by $\spn(X, \rho, \varepsilon)$ the minimal cardinality of $(\rho, \varepsilon)$-spanning subsets of $X$. For any nonempty finite subset $F$ of $\Gamma$, we define a new metric $\rho_F$ on $X$ by $\rho_F(x, y)=\max_{s\in F}\rho(sx, sy)$.

The case $\Gamma=\Zb$ of the following lemma is \cite[Theorem 7.11]{Walters}, whose proof extends to amenable group case easily.

\begin{lemma} \label{L-expansive}
Suppose that the action $\Gamma\curvearrowright X$ is expansive, and let $\kappa$ be an expansive constant with respect to a compatible metric $\rho$ on $X$. Then the following holds.
\begin{enumerate}
\item For any finite open cover $\cU$ of $X$ such that each item of $\cU$ has $\rho$-diameter at most $\kappa$, one has $\htopol(X)=\htopol(\cU)$.
\item For any $0<\varepsilon<\kappa/4$, one has
$$\htopol(X)=\lim_{n\to \infty}\frac{1}{|F_n|}\log \sep(X, \rho_{F_n}, \varepsilon)=\lim_{n\to \infty}\frac{1}{|F_n|}\log \spn(X, \rho_{F_n}, \varepsilon).$$
\end{enumerate}
\end{lemma}

Let $\mu$ be a $\Gamma$-invariant Borel probability measure on $X$. For each finite Borel partition $\cP$ of $X$, one defines the Shannon entropy
$$H_\mu(\cP)=\sum_{P\in \cP}-\mu(P)\log \mu(P),$$
where the convention is $0\log 0=0$, and the dynamical entropy
$$h_\mu(\cP)=\lim_{n\to \infty}\frac{1}{|F_n|}H_\mu(\bigvee_{s\in F_n}s^{-1}\cP).$$
The {\it measure entropy of the action $\Gamma\curvearrowright (X, \mu)$} is defined as
$$h_\mu(X):=\sup_\cP h_\mu(\cP)$$
for $\cP$ ranging over all finite Borel partitions of $X$.
The action $\Gamma\curvearrowright (X, \mu)$ is said to have {\it completely positive entropy} (CPE) if $h_\mu(\cP)>0$ for every finite Borel partition $\cP$  of $X$ with $H_\mu(\cP)>0$.

The variational principle says that $\htopol(X)=\sup_\mu h_\mu(X)$ for $\mu$ ranging over all $\Gamma$-invariant Borel probability measures of $X$.

\section{Myhill property} \label{S-Myhill}

In this section we prove Theorem~\ref{T-Myhill1}. Throughout this section, we let a countable amenable group $\Gamma$ act on compact metrizable spaces $X$ and $Y$ continuously, and fix a left F{\o}lner sequence $\{F_n\}_{n\in \Nb}$ for $\Gamma$.

\begin{proposition} \label{P-subaction}
Assume that the action $\Gamma\curvearrowright Y$ is expansive and has the weak specification property.
For any nonempty closed $\Gamma$-invariant subset $Z$ of $Y$ with $Z\neq Y$, we have $\htopol(Z)<\htopol(Y)$.
\end{proposition}
\begin{proof}
Take a compatible metric $\rho$ on $Y$. Let $\kappa>0$ be an expansive constant of $\Gamma \curvearrowright Y$ with respect to $\rho$.
Take $y_0\in Y\setminus Z$ and set $\eta=\min(\kappa/10, \rho(y_0, Z))>0$. By the weak specification property, there exists a symmetric finite set $F\subseteq \Gamma$ containing $e_\Gamma$ such that for any finite collection of finite subsets $\{K_j\}_{j\in J}$ of $\Gamma$ satisfying $FK_i\cap K_j=\emptyset$ for all distinct $i, j\in J$ and any collection $\{y_j\}_{j\in J}$ of points in $Y$, there exists $y\in Y$ such that $\rho(sy, sy_j)\le \eta/4$ for all $j\in J$ and $s\in K_j$.

Let $n\in \Nb$. Take a maximal set $K_n\subseteq F_n$ subject to the condition that for any distinct $s, t\in K_n$, one has $s\not \in Ft$. Then $FK_n\supseteq F_n$, and hence
$$|K_n|\ge |F_n|/|F|.$$
Let $A\subseteq K_n$. Take a $(\rho_{F_n\setminus (FA)}, \eta)$-spanning subset $W_{F_n\setminus (FA)}$ of $Z$ with cardinality $\spn(Z, \rho_{F_n\setminus (FA)}, \eta)$, and for each $s\in A$ take a $(\rho_{Fs}, \eta)$-spanning subset $W_{Fs}$ of $Z$ with cardinality $\spn(Z, \rho_{Fs}, \eta)=\spn(Z, \rho_F, \eta)$.
For each $z\in Z$, we can take $z_{F_n\setminus (FA)}\in W_{F_n\setminus (FA)}$ with $\rho_{F_n\setminus (FA)}(z, z_{F_n\setminus (FA)})<\eta$ and $z_{Fs}\in W_{Fs}$ with $\rho_{Fs}(z, z_{Fs})<\eta$ for each $s\in A$. For any $z, z'\in Z$, if $z_{F_n\setminus (FA)}=z_{F_n\setminus (FA)}'$ and $z_{Fs}=z_{Fs}'$ for all $s\in A$, then $\rho_{F_n}(z, z')<2\eta$. 
It follows that
\begin{align*}
 \sep(Z, \rho_{F_n}, 2\eta)&\le |W_{F_n\setminus (FA)}|\cdot \prod_{s\in A}|W_{Fs}|\\
 &= \spn(Z, \rho_{F_n\setminus (FA)}, \eta)\spn(Z, \rho_F, \eta)^{|A|}\\
&\le \sep(Z, \rho_{F_n\setminus (FA)}, \eta)\sep(Z, \rho_F, \eta)^{|A|}.
\end{align*}
Take a $(\rho_{F_n\setminus (FA)}, \eta)$-separated subset $\Omega_A$ of $Z$ with maximal cardinality.
 For each $\omega\in \Omega_A$, take $\omega_A\in Y$ such that
$\rho(s\omega_A, y_0)\le \eta/4$ for all $s\in A$ and $\rho(t\omega_A, t\omega)\le \eta/4$ for all $t\in F_n\setminus (FA)$.
For any distinct $\omega, \omega'\in \Omega_A$, we have  $\rho(t\omega, t\omega')\ge \eta$ for some $t\in F_n\setminus (FA)$, and hence
$$\rho(t\omega_A, t\omega'_A)\ge \rho(t\omega, t\omega')-\rho(t\omega, t\omega_A)-\rho(t\omega', t\omega'_A)\ge \eta-\eta/4-\eta/4=\eta/2.$$
For any distinct $A, B\subseteq K_n$ and any $\omega\in \Omega_A, \omega'\in \Omega_B$, say $s\in A\setminus B$, we have
$$ \rho(s\omega_A, s\omega'_B)\ge \rho(y_0, s\omega')-\rho(s\omega_A, y_0)-\rho(s\omega'_B, s\omega')\ge \eta-\eta/4-\eta/4=\eta/2.$$
Thus the set $\{\omega_A: A\subseteq K_n, \omega\in \Omega_A\}$ is a $(\rho_{F_n}, \eta/2)$-separated subset of $Y$ with cardinality $\sum_{A\subseteq K_n}\sep(Z, \rho_{F_n\setminus (FA)}, \eta)$.
Therefore
\begin{align*}
 \sep(Y, \rho_{F_n}, \eta/2)&\ge \sum_{A\subseteq K_n}\sep(Z, \rho_{F_n\setminus (FA)}, \eta)\\
 &\ge \sum_{A\subseteq K_n}\sep(Z, \rho_{F_n}, 2\eta)\sep(Z, \rho_F, \eta)^{-|A|}\\
 &=\sep(Z, \rho_{F_n}, 2\eta)(1+\sep(Z, \rho_F, \eta)^{-1})^{|K_n|}.
\end{align*}
Thus by Lemma~\ref{L-expansive} we have
\begin{align*}
\htopol(Y)&=\lim_{n\to \infty}\frac{1}{|F_n|}\log \sep(Y, \rho_{F_n}, \eta/2)\\
&\ge \lim_{n\to \infty}\frac{1}{|F_n|}\log\sep(Z, \rho_{F_n}, 2\eta)+ \limsup_{n\to \infty}\frac{1}{|F_n|}\log (1+\sep(Z, \rho_F, \eta)^{-1})^{|K_n|}\\
&\ge \htopol(Z)+\frac{1}{|F|}\log (1+\sep(Z, \rho_F, \eta)^{-1})\\
&> \htopol(Z)
\end{align*}
as desired.
\end{proof}

Proposition~\ref{P-subaction} was proved before under the further assumption that $Y$ is a subshift of finite type by Fiorenzi in the proof of \cite[Proposition 4.6]{Fiorenzi03}, and under the further assumption that $Y$ is a subshift by Ceccherini-Silberstein and Coornaert \cite[Proposition 4.2]{CC12}.

\begin{proposition} \label{P-factor}
Assume that the action $\Gamma\curvearrowright X$ is expansive and has the weak specification property. Let $\Gamma \curvearrowright Y$ be a factor of $\Gamma\curvearrowright X$ such that $\Gamma \curvearrowright Y$ is expansive and $\htopol(Y)<\htopol(X)$. Then for any homoclinic equivalence class $\Xi$ of $X$,  the factor map $T:X\rightarrow Y$ fails to be injective on $\Xi$.
\end{proposition}
\begin{proof}
Take compatible metrics $\rho_X$ and $\rho_Y$ on $X$ and $Y$ respectively.
Take a common expansive constant $\kappa>0$ for the action $\Gamma\curvearrowright X$ with respect to $\rho_X$ and the action $\Gamma\curvearrowright Y$ with respect to $\rho_Y$.
Write $\varepsilon=\kappa/15$.
As $T$ is continuous and $X$ is compact, there exists $0<\delta<\varepsilon$ such that for any $x, x'\in X$ with $\rho_X(x, x')\le \delta$ one has $\rho_Y(Tx, Tx')\le \varepsilon$.

Fix a point $z\in \Xi$. Since the action $\Gamma\curvearrowright X$ has the weak specification property,   there exists a symmetric finite set $F\subseteq \Gamma$ containing $e_\Gamma$ such that for any
$x\in X$ and any finite set $K\subseteq \Gamma$ there exists $x'\in X$ with $\max_{s\in K}\rho_X(sx', sx)\le \delta$ and $\sup_{s\in \Gamma\setminus (FK)}\rho_X(sx', sz)\le \delta$.
Set $C=\spn(Y, \rho_Y, \varepsilon)$.

By Lemma~\ref{L-expansive} we have
\begin{align*}
\htopol(X)=\lim_{n\to \infty}\frac{1}{|F_n|}\log \sep(X, \rho_{X, F_n}, 3\varepsilon),
\end{align*}
and
\begin{align*}
\htopol(Y)=
\lim_{n\to \infty}\frac{1}{|F_n|}\log \spn(Y, \rho_{Y, F_n}, \varepsilon).
\end{align*}

Take $\eta>0$ with $\htopol(X)>3\eta+\htopol(Y)$. When $n$ is large enough, we have
$\spn(Y, \rho_{Y, F_n}, \varepsilon)\le e^{|F_n|(\htopol(Y)+\eta)}$ and
$$\sep(X, \rho_{X, F_n}, 3\varepsilon)\ge e^{|F_n|(\htopol(X)-\eta)}\ge  e^{\eta |F_n|}\spn(Y, \rho_{Y, F_n}, \varepsilon).$$
Let $\Omega_n$ be a $(\rho_{X, F_n}, 3\varepsilon)$-separated subset of $X$ with maximum cardinality, and $\Lambda_n$ be a $(\rho_{Y, F_n}, \varepsilon)$-spanning subset of $Y$ with minimal cardinality. Then there exist a set $\Omega'_n\subseteq \Omega_n$ with $|\Omega'_n|\ge e^{\eta |F_n|}$ and $y\in \Lambda_n$ such that $\max_{s\in F_n}\rho_Y(sTx, sy)\le \varepsilon$ for all $x\in \Omega'_n$.
For each $x\in \Omega'_n$, take $x'\in X$
with $\max_{s\in F_n}\rho_X(sx', sx)\le \delta<\varepsilon$ and $\sup_{s\in \Gamma\setminus (FF_n)}\rho_X(sx', sz)\le \delta$.
Then $\sup_{s\in \Gamma\setminus (FF_n)}\rho_X(sx', sz)\le \kappa$, and hence by \cite[Lemma 6.2]{CL} the pair $(x', z)$ is homoclinic.  Thus $x'\in \Xi$.
By our choice of $\delta$, we get
\begin{align*}
\max_{s\in F_n}\rho_Y(sTx', sy)&\le \max_{s\in F_n}\rho_Y(sTx', sTx)+\max_{s\in F_n}\rho_Y(sTx, sy)\\
&\le  \max_{s\in F_n}\rho_Y(T(sx'), T(sx))+\varepsilon\\
&\le \varepsilon+\varepsilon=2\varepsilon,
\end{align*}
and
$$\sup_{s\in \Gamma\setminus (FF_n)}\rho_Y(sTx', sTz)=\sup_{s\in \Gamma\setminus (FF_n)}\rho_Y(T(sx'), T(sz))\le \varepsilon.$$
 Now take a set $\Omega''_n\subseteq \Omega'_n$ with $|\Omega''_n|\ge |\Omega'_n|/C^{|FF_n\setminus F_n|}$ such that for any $x, \omega\in \Omega''_n$ we have $\max_{s\in FF_n\setminus F_n}\rho_Y(sTx', sT\omega')<2\varepsilon$. Then for any $x, \omega\in \Omega''_n$ we have
$\sup_{s\in \Gamma}\rho_Y(sTx', sT\omega')\le 4\varepsilon$, whence $Tx'=T\omega'$.
If $x\neq \omega$, then
$$ \max_{s\in F_n}\rho_X(sx', s\omega')\ge \max_{s\in F_n}\rho_X(sx, s\omega)-\max_{s\in F_n}\rho_X(sx, sx')-\max_{s\in F_n}\rho_X(s\omega, s\omega')>\varepsilon,$$
and hence $x'\neq \omega'$.
When $n$ is large enough, we have $C^{|FF_n\setminus F_n|}<e^{\eta |F_n|}$, and hence $|\Omega''_n|>1$.
\end{proof}

Proposition~\ref{P-factor} was proved before under the further assumption that $X$ is a subshift of finite type and  $Y$ is a subshift by Fiorenzi \cite[Proposition 4.5]{Fiorenzi03},
and under the further assumption that $X$ and $Y$ are subshifts by Ceccherini-Silberstein and Coornaert \cite[Theorem 5.1]{CC12}.

From Propositions~\ref{P-subaction} and \ref{P-factor} we obtain:

\begin{theorem} \label{T-Myhill}
Let $\Gamma\curvearrowright X$ and $\Gamma\curvearrowright Y$ be expansive actions with the weak specification property. Assume that $\htopol(X)=\htopol(Y)$. Then every pre-injective continuous $\Gamma$-equivariant map $X\rightarrow Y$ is surjective.
\end{theorem}

Theorem~\ref{T-Myhill} was proved before under the further assumption that $X$ is a subshift of finite type and  $Y$ is a subshift by Fiorenzi \cite[Theorem 4.7]{Fiorenzi03},
and under the further assumption that $X$ and $Y$ are subshifts by Ceccherini-Silberstein and Coornaert \cite[Corollary 5.2]{CC12}.

Theorem~\ref{T-Myhill1} follows from Theorem~\ref{T-Myhill} directly.

\section{Moore Property} \label{S-Moore}

In this section we prove Theorem~\ref{T-Moore1}. Throughout this section $\Gamma$ will be a countable amenable group and $\{F_n\}_{n\in \Nb}$ will be a left F{\o}lner sequence of $\Gamma$.

\begin{lemma} \label{L-measure}
Let $\Gamma$ act on compact metrizable groups $X$ and $Y$ by automorphisms. Denote by $\mu_X$ and $\mu_Y$ the normalized Haar measure of $X$ and $Y$ respectively. Suppose that the action $\Gamma\curvearrowright (X, \mu_X)$ has CPE. Also assume that $\htopol(X)=\htopol(Y)<+\infty$. Let $T: X\rightarrow Y$ be a $\Gamma$-equivariant continuous surjective map. Then $T_*\mu_X=\mu_Y$.
\end{lemma}
\begin{proof}
Since $\Gamma$ is amenable and $T$ is surjective, there is a $\Gamma$-invariant Borel probability measure $\nu$ on $X$ satisfying $T_*\nu=\mu_Y$. Then
$$\htopol(X)\ge h_\nu(X)\ge h_{\mu_Y}(Y)=\htopol(Y),$$
where the equality is from \cite[Theorem 2.2]{Den} \cite[Proposition 13.2]{KL}. Thus $h_\nu(X)=\htopol(X)$.
Since the action $\Gamma\curvearrowright (X, \mu_X)$ has CPE and $h_{\mu_X}(X)\le \htopol(X)<+\infty$, by \cite[Theorem 8.6]{CL} we have $h_{\nu'}(X)<h_{\mu_X}(X)$ for every $\Gamma$-invariant Borel probability measure $\nu'$ on $X$ different from $\mu_X$.  Therefore $\nu=\mu_X$. Thus
 $T_*\mu_X=\mu_Y$.
\end{proof}

\begin{theorem} \label{T-Moore}
Let $\Gamma\curvearrowright X$ be an expansive algebraic action with CPE with respect to the normalized Haar measure and $\Gamma\curvearrowright Y$ be an expansive  action on a compact metrizable group by automorphisms. Assume that $\htopol(X)=\htopol(Y)$. Then every surjective continuous $\Gamma$-equivariant  map $T:X\rightarrow Y$ is pre-injective.
\end{theorem}
\begin{proof}When $\Gamma$ is finite, we have $|X|=|\Gamma|\htopol(X)=|\Gamma|\htopol(Y)=|Y|<+\infty$. Then $T$ is actually injective. Thus we may assume that $\Gamma$ is infinite.

Assume that $T$ is not pre-injective. Then there is a homoclinic pair $(x, \omega)\in X^2$ such that $x\neq \omega$ and $Tx=T\omega$. We get $\omega-x\in \Delta(X)$.

By Lemma~\ref{L-metric} we can find a compatible translation-invariant metric $\rho_X$ on $X$  such that $\sum_{s\in \Gamma}\rho_X(sx', 0_X)<+\infty$ for all $x'\in \Delta(X)$. We also take a compatible translation-invariant metric $\rho_Y$ on $Y$.

Denote by $\mu_X$ and $\mu_Y$ the normalized Haar measure of $X$ and $Y$ respectively.

Take a common expansive constant $\kappa>0$ for the action $\Gamma\curvearrowright X$ with respect to $\rho_X$ and the action $\Gamma\curvearrowright Y$ with respect to $\rho_Y$.
Let $0<\varepsilon<\kappa/4$. By Lemma~\ref{L-expansive} we have
$$\htopol(Y)=\lim_{n\to \infty}\frac{1}{|F_n|}\log \sep(Y, \rho_{Y, F_n}, \varepsilon).$$
For each $n\in \Nb$, denote by $D_n$ the set of $y\in Y$ satisfying $\max_{s\in F_n}\rho_Y(sy, e_Y)<\varepsilon/2$, where $e_Y$ denotes the identity element of $Y$, and take a
$(\rho_{Y, F_n}, \varepsilon)$-separated subset $W_n$ of $Y$ with maximal cardinality. For any distinct $y, z\in W_n$, one has $(yD_n)\cap (zD_n)=\emptyset$.  Thus $\mu_Y(D_n)|W_n|\le 1$, and hence
$$1/\mu_Y(D_n)\ge |W_n|.$$

Since $T$ is continuous and $X$ is compact, we can find $0<\delta<\min(\kappa/4, \rho_X(\omega-x, 0_X)/2)$ such that for any $x_1, x_2\in X$ with $\rho_X(x_1, x_2)\le 2\delta$, one has $\rho_Y(Tx_1, Tx_2)<\varepsilon/16$.
Since $\sum_{s\in \Gamma}\rho_X(s(\omega-x), 0_X)<+\infty$, we can find a symmetric finite set $F\subseteq \Gamma$ containing $e_\Gamma$ such that
$$\sum_{s\in \Gamma\setminus F}\rho_X(s(\omega-x), 0_X)<\delta.$$
Take $0<\tau<\delta$ such that for any $x_1, x_2\in X$ with $\rho_X(x_1, x_2)\le \tau$, one has $\max_{s\in F}\rho_X(sx_1, sx_2)\le \delta$.

Denote by $B$ the set of all $x'\in X$ satisfying $\rho_X(x', x)\le \tau$. Then $\mu_X(B)>0$. Since the action $\Gamma\curvearrowright (X, \mu_X)$ has CPE, $\mu_X$ is ergodic. By the mean ergodic theorem \cite[Theorem 4.22]{KL}, we have $\|\frac{1}{|F_n|}\sum_{s\in F_n}s^{-1}1_{B}-\mu_X(B)\|_2\to 0$ as $n\to \infty$, where $1_B$ denotes the characteristic function of $B$. For each $n\in \Nb$, denote by $X_n$ the set of $x'\in X$ satisfying $|\{s\in　F_n: sx'\in B\}|\ge |F_n|\mu_X(B)/2$. Then $X_n$ is closed and  $\mu_X(X_n)\to 1$ as $n\to \infty$.

Since the action $\Gamma\curvearrowright X$ is expansive, by Lemma~\ref{L-expansive} we have $\htopol(X)<+\infty$.
By Lemma~\ref{L-measure} we have  $T_*\mu_X=\mu_Y$. Then $\mu_Y(T(X_n))\ge \mu_X(X_n)$, and hence $\mu_Y(T(X_n))\to 1$ as $n\to \infty$. Take a maximal $(\rho_{Y, F_n}, \varepsilon/2)$-separated subset $W_n'$ of $T(X_n)$. Then $yD_n$ for $y\in W_n'$ covers $T(X_n)$. Thus $\mu_Y(T(X_n))\le \mu_Y(D_n)|W_n'|$, whence
 $$|W'_n|\ge \mu_Y(T(X_n))/\mu_Y(D_n)\ge \mu_Y(T(X_n))|W_n|.$$
 Take a subset $V_n$ of $X_n$ with $|V_n|=|W'_n|$ and $T(V_n)=W'_n$. Then
\begin{align*}
\liminf_{n\to \infty}\frac{1}{|F_n|}\log |V_n|&=\liminf_{n\to \infty}\frac{1}{|F_n|}\log |W'_n|\\
&\ge  \liminf_{n\to \infty}\frac{1}{|F_n|}\log \mu_Y(T(X_n))+\liminf_{n\to \infty}\frac{1}{|F_n|}\log |W_n|\\
&=\htopol(Y).
\end{align*}

For any $n\in \Nb$ and $v\in V_n$, write $E_{n, v}=\{s\in F_n: sv\in B\}$ and take a maximal set $E'_{n, v}\subseteq E_{n, v}$ subject to the condition $Ft\cap Fs=\emptyset$ for all distinct $s, t\in E'_{n, v}$. Then $F^2E'_{n, v}\supseteq E_{n, v}$, and hence
$$|E'_{n, v}|\ge |E_{n, v}|/|F|^2\ge |F_n|\mu_X(B)/(2|F|^2).$$
For each set $A\subseteq E'_{n, v}$, define
$$v_A=v+\sum_{s\in A}s^{-1}(\omega-x)\in X.$$
We claim that $Tv_A=Tv$. It suffices to show $\rho_Y(tTv_A, tTv)=\rho_Y(T(tv_A), T(tv))<\varepsilon/8$ for all $t\in \Gamma$. Let $t\in \Gamma$. If $t\not \in FA$, then
$$\rho_X(tv_A, tv)=\rho_X(\sum_{s\in A}ts^{-1}(\omega-x), 0_X)\le \sum_{s\in \Gamma\setminus F}\rho_X(s(\omega-x), 0_X)<\delta,$$
and hence $\rho_Y(T(tv_A), T(tv))<\varepsilon/16$. Now consider the case $t\in FA$. Say $t=\gamma s'$ for some $\gamma \in F$ and $s'\in A$. Then
\begin{align*}
\rho_X(tv_A, \gamma \omega)&=\rho_X(\gamma s'v+\sum_{s\in A\setminus \{s'\}}ts^{-1}(\omega-x), \gamma x)\\
&\le \rho_X(\gamma s'v, \gamma x)+\sum_{s\in \Gamma\setminus F}\rho_X(s(\omega-x), 0_X)<2\delta,
\end{align*}
and
$$\rho_X(tv, \gamma x)=\rho_X(\gamma s'v, \gamma x)\le \delta.$$
Therefore
\begin{align*}
\rho_Y(T(tv_A), \gamma Tx)=\rho_Y(T(tv_A), T(\gamma \omega))<\varepsilon/16,
\end{align*}
and
$$\rho_Y(T(tv), \gamma Tx)=\rho_Y(T(tv), T(\gamma x))<\varepsilon/16.$$
Consequently,
$$\rho_Y(T(tv_A), T(tv))\le \rho_Y(T(tv_A), \gamma Tx)+\rho_Y(T(tv), \gamma Tx)<\varepsilon/8.$$
This proves our claim.

Write $V^\dag_n:=\{v_A: v\in V_n, A\subseteq E'_{n, v}\}$.
For any $v\in V_n$ and distinct $A, A'\subseteq E'_{n, v}$, say $t\in A\setminus A'$, we have
\begin{align*}
\rho_X(tv_A, tv_{A'})&\ge
\rho_X(\omega-x, 0_X)-\sum_{s\in (A\setminus \{t\})\Delta A'}\rho_X(ts^{-1}(\omega-x), 0_X)\\
&\ge \rho_X(\omega-x, 0_X)-\sum_{s\in \Gamma\setminus F}\rho_X(s(\omega-x), 0_X)\\
&\ge \delta.
\end{align*}
For any distinct $v, z\in V_n$, and $A\subseteq E'_{n, v}$ and $A'\subseteq E'_{n, z}$, we have
$$ \max_{s\in F_n}\rho_Y(Tsv_A, Tsz_{A'})=\max_{s\in F_n}\rho_Y(sTv_A, sTz_{A'})=\max_{s\in F_n}\rho_Y(sTv, sTz)\ge \varepsilon/2,$$
whence  $\max_{s\in F_n}\rho_X(sv_A, sz_{A'})> 2\delta$. Thus $V^\dag_n$ is $(\rho_{X, F_n}, \delta)$-separated, and
$$|V^\dag_n|\ge |V_n|2^{|F_n|\mu_X(B)/(2|F|^2)}.$$
Therefore by Lemma~\ref{L-expansive} we have
\begin{align*}
\htopol(X)&= \lim_{n\to \infty}\frac{1}{|F_n|}\log \sep(X, \rho_{X, F_n}, \delta)\\
&\ge \liminf_{n\to \infty}\frac{1}{|F_n|}\log |V_n|+\frac{\mu_X(B)\log 2}{2|F|^2}\\
&\ge \htopol(Y)+\frac{\mu_X(B)\log 2}{2|F|^2},
\end{align*}
which is a contradiction to the hypothesis $\htopol(X)=\htopol(Y)$. Thus $T$ is pre-injective.
\end{proof}

Now Theorem~\ref{T-Moore1} follows from Theorem~\ref{T-Moore} directly.

\section{Weak specification and independence} \label{S-ind}

In this section we discuss implications of weak specification to combinatorial independence and prove Corollary~\ref{C-CPE}.

Let a countable (not necessarily amenable) group $\Gamma$ act on a compact metrizable space $X$ continuously. Let ${\bf A}=(A_1, \dots, A_k)$ be a tuple of subsets of $X$. A nonempty finite set $K'\subseteq \Gamma$ is called {\it an independence set} for ${\bf A}$ if $\bigcap_{s\in K'}s^{-1}A_{\omega(s)}\neq \emptyset$ for all maps $\omega: K'\rightarrow \{1, \dots, k\}$ \cite[Definition 8.7]{KL}. For any nonempty finite set $K\subseteq \Gamma$ write $\varphi_{\bf A}(K)$ for the maximal cardinality of independence sets $K'$ of $\bf A$ satisfying $K'\subseteq K$.
The  {\it  independence density} of $\bf A$ is defined as
$$I({\bf A}):=\inf_K\frac{\varphi_{\bf A}(K)}{|K|},$$
where $K$ ranges over all nonempty finite subsets of $\Gamma$ \cite[Definition 3.1]{KL13b}.
A tuple ${\bf x}=(x_1, \dots, x_k)\in X^k$ is called {\it an orbit IE-tuple} if for every product neighborhood $U_1\times \cdots \times U_k$ of $\bf x$, the tuple $(U_1, \dots, U_k)$ has positive independence density \cite[Definition 3.2]{KL13b}.

\begin{proposition} \label{P-UPE}
For any action $\Gamma\curvearrowright X$  with the weak specification property, every tuple is an orbit IE-tuple.
\end{proposition}
\begin{proof}
Let $k\in \Nb$ and $x_1, \dots, x_k\in X$. We shall show that $(x_1, \dots, x_k)\in X^k$ is an orbit IE-tuple.

Let $\rho$ be a compatible metric on $X$, and let $\varepsilon>0$. Denote by $D_i$ the set of all $x\in X$ satisfying $\rho(x, x_i)\le \varepsilon$. By the weak specification property there is some symmetric finite subset $F$ of $\Gamma$ containing $e_\Gamma$ such that for any finite collection $\{F_j\}_{j\in J}$ of finite subsets of $\Gamma$ satisfying $FF_i\cap F_j=\emptyset$ for all distinct $i, j\in J$ and any collection $\{y_j\}_{j\in J}$ of points in $X$, there is some $y\in X$ such that $\rho(y, sy_j)\le \varepsilon$ for all $j\in J$ and $s\in F_j$.

Let $K$ be a nonempty finite subset of $\Gamma$. Take a maximal subset $K'$ of $K$ subject to the condition that $s\not\in Ft$ for all distinct $s, t\in K'$. Then $FK'\supseteq K$, and hence $|K'|\ge |K|/|F|$. Let $\omega$ be a map $K'\rightarrow \{1, \dots, k\}$. Then there is some $x\in X$ such that $\rho(sx, x_{\omega(s)})\le \varepsilon$ for all $s\in K'$. Thus $K'$ is an independence set for the tuple $(D_1, \dots, D_k)$. It follows that $(D_1, \dots, D_k)$ has independence density at least $1/|F|$. Therefore $(x_1, \dots, x_k)$ is an orbit IE-tuple.
\end{proof}

Now we consider the case $\Gamma$ is amenable.
An action $\Gamma\curvearrowright X$ has positive entropy if and only if there is at least one non-diagonal orbit IE-pair in $X^2$ \cite[Definition 12.5, Theorem 12.19]{KL}. Thus we get

\begin{corollary} \label{C-factor}
For any countable amenable group, every continuous action on a compact metrizable space with the weak specification property and more than one point has positive entropy.
\end{corollary}

Corollary~\ref{C-factor} was proved before under the further assumption of subshifts by Ceccherini-Silberstein and Coornaert \cite[Proposition 4.5]{CC12}.

For any action of a countable amenable group $\Gamma$ on a compact metrizable group $X$ by automorphisms, every pair in $X^2$ is an orbit IE-pair if and only if the action has CPE with respect to the normalized Haar measure \cite[Definition 12.5]{KL} \cite[Theorem 7.3, Corollary 8.4]{CL}. Thus we get

\begin{corollary} \label{C-CPE}
Every weak specification action of a countable amenable group on a compact metrizable group by automorphisms has  CPE with respect to the normalized Haar measure.
\end{corollary}

\appendix

\section{Subshifts with weak specification} \label{S-subshifts}

Let $\Gamma$ be a countable (not necessarily amenable) group, and let $A$ be a finite set. We consider the shift action $\Gamma\curvearrowright A^\Gamma$ given by $(sx)_t=x_{s^{-1}t}$ for all $x\in A^G$ and $s, t\in \Gamma$. A closed $\Gamma$-invariant subset $X$ of $A^\Gamma$ is called {\it strongly irreducible} \cite[Definition 4.1]{Fiorenzi03} if there exists a nonempty symmetric finite set $F\subseteq \Gamma$ such that for any finite sets $F_1, F_2\subseteq \Gamma$ with $F_1F\cap F_2=\emptyset$ and any $x_1, x_2\in X$, there exists $x\in X$ such that $x=x_1$ on $F_1$
and $x=x_2$ on $F_2$.

\begin{proposition} \label{P-subshifts}
For any closed $\Gamma$-invariant subset $X$ of $A^\Gamma$, $X$ is strongly irreducible if and only if it has the weak specification property.
\end{proposition}
\begin{proof} Take a compatible metric $\rho$ of $X$.

Suppose that $X$ is strongly irreducible. Let $F\subseteq \Gamma$ witness the strong irreducibility of $X$. By induction it is easy to see that for any finite collection $\{F_j\}_{j\in J}$ of finite subsets of $\Gamma$ satisfying $F_iF\cap F_j=\emptyset$ for all distinct $i, j\in J$ and any collection $\{x_j\}_{j\in J}$ of points in $X$, there is some $x\in X$ such that $x=x_j$ on $F_j$ for all $j\in J$. Let $\varepsilon>0$. Then there is some nonempty finite subset $K$ of $\Gamma$ such that for any
$y, z\in X$ with $y=z$ on $K$ one has $\rho(y, z)\le \varepsilon$. Now let $\{F_j\}_{j\in J}$ be a finite collection of finite subsets of $\Gamma$ satisfying $(KFK^{-1})F_i\cap F_j=\emptyset$ for all distinct $i, j\in J$ and $\{x_j\}_{j\in J}$ be a collection of points in $X$. Then $(F_i^{-1}K)F\cap (F_j^{-1}K)=\emptyset$ for all distinct $i, j\in J$. Thus there is some $x\in X$ satisfying $x=x_j$ on $F_j^{-1}K$ for all $j\in J$. For any $j\in J$ and $s\in F_j$, we get $sx=sx_j$ on $K$, and hence $\rho(sx, sx_j)\le \varepsilon$. Therefore $X$ has the weak specification property.

Conversely suppose that $X$ has the weak specification property. Take $\varepsilon>0$ such that any two points $y, z\in X$ satisfying $\rho(y, z)\le \varepsilon$ must coincide at $e_\Gamma$. Then there is some nonempty symmetric finite subset $F$ of $\Gamma$ such that for any finite subsets $F_1$ and $F_2$ of $\Gamma$ satisfying  $FF_1\cap F_2=\emptyset$ and any points $x_1, x_2\in X$, there is some $x\in X$ such that $\rho(sx, sx_j)\le \varepsilon$ for all $j=1, 2$ and $s\in F_j$. Now let $F_1$ and $F_2$ be finite subsets of $\Gamma$ satisfying $F_1F\cap F_2=\emptyset$ and $x_1, x_2\in X$. Then $FF_1^{-1}\cap F_2^{-1}=\emptyset$. Thus there is some $x\in X$ such that $\rho(sx, sx_j)\le \varepsilon$ for all $j=1, 2$ and $s\in F_j^{-1}$. Then $sx=sx_j$ at $e_\Gamma$, which means $x=x_j$ at $s^{-1}$. Therefore $X$ is strongly irreducible.
\end{proof}


\end{document}